\documentclass[journal]{IEEEtran}

\usepackage[T1]{fontenc}
\usepackage{amsmath,amssymb,amsthm}
\usepackage{cite}
\usepackage{graphicx}
\usepackage[hidelinks]{hyperref}

\hypersetup{
    pdftitle={A Globally Asymptotically Stable Planar Homogeneous Polynomial Vector Field With No Polynomial Lyapunov Function},
    pdfauthor={Jun Liu and Maxwell Fitzsimmons}
}

\newtheorem{theorem}{Theorem}[section]

\newtheorem{lemma}[theorem]{Lemma}

\theoremstyle{definition}
\newtheorem{definition}[theorem]{Definition}
\theoremstyle{remark}
\newtheorem{remark}[theorem]{Remark}

\newcommand{\R}{\mathbb{R}}
\newcommand{\Lie}{\mathcal{L}_f}
\newcommand{\dd}{\mathop{}\!\mathrm{d}}
\newcommand{\e}{\mathrm{e}}
\newcommand{\norm}[1]{\lVert #1\rVert}

\title{A Globally Asymptotically Stable \\Planar Homogeneous Polynomial Vector Field \\With No Polynomial Lyapunov Function}
\author{Jun~Liu and Maxwell~Fitzsimmons%
\thanks{This work was supported in part by the Natural Sciences and
Engineering Research Council of Canada and the Canada Research Chairs
Program.}%
\thanks{Jun Liu and Maxwell Fitzsimmons are with the Department of Applied Mathematics, Faculty of
Mathematics, University of Waterloo, Waterloo, ON N2L 3G1, Canada (e-mail:
\texttt{j.liu@uwaterloo.ca}; \texttt{mfitzsimmons@uwaterloo.ca}).}}

\begin{document}

\maketitle

\begin{abstract}
We disprove the conjecture that every globally asymptotically stable
homogeneous polynomial vector field admits a homogeneous polynomial Lyapunov
function.  The counterexample is a planar homogeneous cubic polynomial vector
field with integer coefficients.  It admits no positive definite homogeneous
polynomial with nonpositive Lie derivative and, more strongly, no
real-analytic Lyapunov function even locally.  Nevertheless, it has an explicit
degree-two homogeneous Lyapunov function that is radially unbounded,
continuously differentiable everywhere, and smooth away from the origin. 
We also provide a machine-checked Lean 4 formalization of the main result. 
\end{abstract}


\section{Introduction}\label{sec:introduction}

\IEEEPARstart{P}{olynomial} vector fields have finite algebraic descriptions,
but it is not known whether those descriptions suffice to decide asymptotic
stability of an equilibrium.  In a contribution to the proceedings of the
1974 AMS symposium on developments arising from Hilbert's problems, Arnold
asked the broader question of whether stability is decidable for polynomial
vector fields of fixed degree with rational coefficients
\cite[p.~59]{arnold1976problems}; he was later quoted as conjecturing that
``there is no algorithm for some sufficiently high degree and dimension''
\cite{daCostaDoria1993}.  For any fixed candidate degree, Tarski's
quantifier-elimination theorem \cite{tarski1998decision} can decide in finite
time whether there exists
a polynomial of at most that degree satisfying the strict Lyapunov inequalities
on some punctured neighbourhood of the equilibrium.  Classical converse
Lyapunov theorems guarantee
smooth, but not necessarily polynomial, Lyapunov functions
\cite{kellett2015converse}.  A converse
guaranteeing polynomial Lyapunov functions with a computable degree bound
would, together with Tarski's theorem, answer Arnold's question affirmatively.

Homogeneous polynomial vector fields were a natural test case.  A homogeneous
vector field is determined by its restriction to the unit sphere, local and
global asymptotic stability are equivalent, and asymptotic stability guarantees
a homogeneous Lyapunov function of any prescribed finite differentiability
class \cite[Prop.~1 and Thm.~2]{rosier1992}.  Despite this structure, deciding
asymptotic stability is strongly NP-hard already for homogeneous cubic vector
fields \cite{ahmadi2012difficulty}.  Ahmadi raised the polynomial Lyapunov
question in his doctoral thesis \cite[Sec.~4.6]{ahmadi2011thesis} and later
formulated the homogeneous polynomial Lyapunov converse conjecture: \emph{every
globally asymptotically stable homogeneous polynomial vector field admits a
homogeneous polynomial Lyapunov function}
\cite{ahmadi2012difficulty}.  Ahmadi and Parrilo subsequently restated the
conjecture as an open problem \cite[Secs.~4.2 and~6]{ahmadiParrilo2013}.

Subsequent work clarified the surrounding picture without resolving the
conjecture.  Ahmadi and El Khadir
proved that every asymptotically stable continuously differentiable homogeneous
vector field admits a homogeneous rational Lyapunov function and exhibited a
family with a fixed low-degree rational certificate but no uniform bound on
polynomial Lyapunov degree
\cite[Thm.~III.1 and Prop.~VI.1]{ahmadiElKhadir2020}.  They left open whether
the denominator could always be removed
\cite[Sec.~VII, Question~1]{ahmadiElKhadir2020}.  Complementing this algebraic
line, Fitzsimmons and Liu proposed a neural architecture with a universal
approximation guarantee for homogeneous Lyapunov functions
\cite{fitzsimmonsLiu2024}.  Previous work also showed
that neither the minimum polynomial Lyapunov degree
\cite{ahmadi2012difficulty,ahmadiParrilo2013} nor the numerator degree of a
rational Lyapunov function \cite[Prop.~V.1]{ahmadiElKhadir2020} can be bounded
solely in terms of state dimension and vector-field degree.  Earlier counterexamples showing non-existence of global
polynomial or even local real-analytic Lyapunov functions did not settle the
conjecture because their vector fields were nonhomogeneous
\cite[Prop.~5.2]{bacciottiRosier2005},
\cite{ahmadiKrsticParrilo2011,ahmadiElKhadir2018local}.

This paper disproves the homogeneous polynomial Lyapunov converse conjecture.
We give a planar homogeneous cubic polynomial vector field with integer
coefficients whose origin is globally asymptotically stable.  It has an explicit positive
definite, radially unbounded, $2$-homogeneous Lyapunov function in $C^1(\R^2)$, smooth away
from the origin, but no positive definite homogeneous polynomial with
nonpositive Lie derivative.  A leading-term argument rules out every
real-analytic Lyapunov function even locally.  The non-existence of homogeneous
polynomial Lyapunov functions persists on an open region of a two-parameter
family, showing robustness under parameter perturbations within that family.
The example is also minimal with respect to both state dimension and
vector-field degree:
$V(x)=x^2$ is a strict Lyapunov function for every asymptotically stable
scalar homogeneous polynomial system, linear systems admit quadratic Lyapunov
functions, and homogeneous vector fields of even degree cannot be
asymptotically stable \cite[Sec.~17]{hahn1967}.

The non-existence proof is carried out on the unit circle, where a homogeneous
polynomial candidate becomes a positive trigonometric polynomial.
Fej\'er--Riesz factorisation of the resulting positive trigonometric
polynomial, together with the Lyapunov inequality, yields complementary,
degree-independent bounds on the relevant Fourier coefficients.  The angular
forcing of the vector field makes these bounds incompatible.  A leading-term
argument then transfers the conclusion to local real-analytic Lyapunov
functions.

\section{Problem and Main Result}

Consider the autonomous system
\begin{equation}
    \dot z=f(z), \qquad z\in\R^n, \qquad f(0)=0.
\end{equation}
The vector field $f:\R^n \to \R^n$ is homogeneous of degree $d>0$ if
$f(\lambda z)=\lambda^d f(z)$ for every $\lambda>0$.  A function
$V\colon\R^n\to\R$ is homogeneous of degree $m>0$ if
$V(\lambda z)=\lambda^m V(z)$.  If $V$ is differentiable, its Lie derivative along $f$ is
\[
    \Lie V(z)=\nabla V(z)\cdot f(z).
\]
Here $\nabla V=(\partial V/\partial z_1,\ldots,\partial V/\partial z_n)^\top$
is the gradient of $V$.  We say that $V$ is \emph{positive definite} if $V(0)=0$
and $V(z)>0$ for $z\neq0$, and it is \emph{radially unbounded} if
$V(z)\to\infty$ as $\norm{z}\to\infty$.

\begin{definition}
A weak polynomial Lyapunov function is a polynomial $P$ such that $P(0)=0$,
$P(z)>0$ for every $z\neq0$, and $\Lie P(z)\leq0$ for every $z\neq0$.
It is strict if $\Lie P(z)<0$ for every $z\neq0$.
\end{definition}

We study the homogeneous cubic vector field $f=(f_1,f_2)$ given by
\begin{align}
    f_1(x,y)&=4x^3-x^2y-6xy^2-y^3, \label{eq:field-x}\\
    f_2(x,y)&=x^3+4x^2y+xy^2-6y^3. \label{eq:field-y}
\end{align}

\begin{theorem}[Main result]\label{thm:main}
For the vector field \eqref{eq:field-x}--\eqref{eq:field-y}, the following
statements hold.
\begin{enumerate}
    \item The origin is globally asymptotically stable.

    \item Define $H(0,0)=0$ and, for $(x,y)\neq(0,0)$, let
    \begin{equation}\label{eq:H-cartesian}
        H(x,y)=(x^2+y^2)
        \exp\!\left(-\frac{10xy}{x^2+y^2}\right).
    \end{equation}
    Then $H$ is positive definite, radially unbounded, and
    $2$-homogeneous.  Moreover,
    \begin{align*}
        H&\in C^1(\R^2)\cap C^\infty(\R^2\setminus\{0\}),\\
        \Lie H&=-2(x^2+y^2)H<0
    \end{align*}
    away from the origin.

    \item No positive definite homogeneous polynomial $P$ satisfies
    $\Lie P\leq0$.  In particular, no homogeneous polynomial Lyapunov
    function exists.

    \item There is no real-analytic function $V$ on any neighbourhood of
    the origin such that, for some $\rho>0$,
    \begin{align*}
        V(0)&=0,\\
        V(z)&>0 &&(0<\norm{z}<\rho),\\
        \Lie V(z)&\leq0 &&(0<\norm{z}<\rho).
    \end{align*}
\end{enumerate}
\end{theorem}

Section~\ref{sec:proof} proves the theorem.  Subsection~\ref{sec:stability}
derives the polar dynamics and verifies the stated properties of $H$, thereby
establishing global asymptotic stability.  Subsection~\ref{sec:fourier}
develops the needed facts about positive
trigonometric polynomials, which Subsection~\ref{sec:obstruction} uses to
rule out homogeneous polynomial Lyapunov functions.  Subsection~\ref{sec:local}
extends the non-existence result to local real-analytic Lyapunov functions.
Some further remarks are collected in Section~\ref{sec:remarks}.

\section{Proof of the Main Theorem}\label{sec:proof}

\subsection{Lyapunov Analysis of Global Asymptotic Stability}\label{sec:stability}

Write $x=r\cos\theta$ and $y=r\sin\theta$, with $r>0$.  For a planar
vector field,
\begin{equation}\label{eq:polar-general}
    \dot r=\frac{xf_1+yf_2}{r},
    \qquad
    \dot\theta=\frac{xf_2-yf_1}{r^2}.
\end{equation}
Direct expansion gives
\begin{align}
    xf_1+yf_2
        &=4x^4-2x^2y^2-6y^4
          =r^4(-1+5\cos2\theta), \label{eq:radial-identity}\\
    xf_2-yf_1
        &=(x^2+y^2)^2=r^4. \label{eq:angular-identity}
\end{align}
Consequently,
\begin{equation}\label{eq:polar-system}
    \dot r=r^3(-1+5\cos2\theta),
    \qquad
    \dot\theta=r^2.
\end{equation}

In polar coordinates, \eqref{eq:H-cartesian} becomes
\begin{equation}\label{eq:H-polar}
    H(r,\theta)=r^2\e^{-5\sin2\theta}.
\end{equation}
It satisfies
\begin{equation}\label{eq:H-bounds}
    \e^{-5}r^2\leq H(r,\theta)\leq\e^5r^2.
\end{equation}
For $r>0$, \eqref{eq:polar-system} gives
\begin{align*}
    \frac{\dot H}{H}
        &=2\frac{\dot r}{r}-10\cos2\theta\,\dot\theta\\
        &=2r^2(-1+5\cos2\theta)-10r^2\cos2\theta\\
        &=-2r^2.
\end{align*}
This proves the derivative identity in Theorem~\ref{thm:main}(2).  The
bounds \eqref{eq:H-bounds} prove positive definiteness and radial
unboundedness.  It remains to verify the regularity of $H$ at the origin.

For a multi-index $\alpha=(\alpha_1,\ldots,\alpha_n)\in\mathbb{Z}_{\geq0}^n$,
write
\begin{align*}
    |\alpha|&=\alpha_1+\cdots+\alpha_n,
    &\alpha!&=\alpha_1!\cdots\alpha_n!,\\
    x^\alpha&=x_1^{\alpha_1}\cdots x_n^{\alpha_n}.
\end{align*}
We use
$\partial^\alpha F=\partial_1^{\alpha_1}\cdots\partial_n^{\alpha_n}F$,
with $\partial^0F=F$.  Let $S^{n-1}=\{x\in \R^n: \norm{x}=1\}$.

\begin{lemma}[Regularity at a cone tip]\label{lem:cone}
Let $k$ be a positive integer and let $h\in C^\infty(S^{n-1})$.  Define
\[
    F(x)=
    \begin{cases}
        \norm{x}^k h\!\left(x/\norm{x}\right),&x\neq0,\\
        0,&x=0.
    \end{cases}
\]
Then
\[
    F\in C^{k-1}(\R^n)\cap C^\infty(\R^n\setminus\{0\}),
\]
and every derivative of order at most $k-1$ vanishes at the origin.
\end{lemma}

\begin{proof}
For $x\neq0$, the chain rule shows that all derivatives
$\partial^\alpha F$ exist and are continuous.  Differentiating
$F(tx)=t^kF(x)$ with respect to $x$ shows that $\partial^\alpha F$ is
homogeneous of degree $k-|\alpha|$.  Hence, with
\[
    C_\alpha=\max_{\hat x\in S^{n-1}}
        |\partial^\alpha F(\hat x)|,
\]
we have
\begin{equation}\label{eq:derivative-bound}
    |\partial^\alpha F(x)|
        =\norm{x}^{k-|\alpha|}
          |\partial^\alpha F(x/\norm{x})|
        \leq C_\alpha\norm{x}^{k-|\alpha|}.
\end{equation}
For $|\alpha|\leq k-1$, extend $\partial^\alpha F$ to the origin by setting
$\partial^\alpha F(0)=0$.  The estimate shows that each extension is
continuous.  For $|\alpha|\leq k-2$,
\[
    \left|
        \frac{\partial^\alpha F(x)-\partial^\alpha F(0)}{\norm{x}}
    \right|
    \leq C_\alpha\norm{x}^{k-|\alpha|-1}\longrightarrow0
    \qquad(x\to0).
\]
Thus each such extension is differentiable at the origin with derivative
zero.  Induction on $|\alpha|$ shows that these extensions are the partial
derivatives of $F$ through order $k-1$, proving the claim.
\end{proof}

Lemma~\ref{lem:cone}, with $k=2$, shows that $H\in C^1(\R^2)$.  Its angular
profile is smooth, so it is smooth away from the origin.  The bounds
\eqref{eq:H-bounds} and the strict derivative inequality now imply global
asymptotic stability by Lyapunov's direct method.  This proves
Theorem~\ref{thm:main}(1)--(2).

\subsection{Preliminaries on Positive Trigonometric Polynomials}%
\label{sec:fourier}

For a $2\pi$-periodic integrable function $u$, write
\begin{equation}\label{eq:fourier-convention}
    \widehat u_j=\frac{1}{2\pi}\int_0^{2\pi}
        u(\theta)\e^{-ij\theta}\,\dd\theta.
\end{equation}
The non-existence proof rests on the strict Fej\'er--Riesz factorisation for
scalar trigonometric polynomials
\cite[Thm.~1.1 and Rem.~1.2]{dumitrescu2017} and a resulting degree-independent
bound on the second Fourier coefficient of a normalised logarithmic derivative.
We record both facts here.  In the factorisation, we use the factor with no
zeros in the closed unit disk; an elementary proof is included to make this
choice explicit.

\begin{lemma}[Strict trigonometric factorisation]\label{lem:factorization}
Let $q$ be a real-valued trigonometric polynomial of actual Fourier degree
$L$, written as
\[
    q(\varphi)=\sum_{j=-L}^{L}c_j\e^{ij\varphi}.
\]
Suppose that $q(\varphi)>0$ for every $\varphi\in\R$.  Then there is a
polynomial $Q$ of degree $L$, with no zeros in the closed unit disk, such that
\[
    q(\varphi)=\lvert Q(\e^{i\varphi})\rvert^2.
\]
\end{lemma}

\begin{proof}
The case $L=0$ is immediate.  Suppose $L\geq1$.  Since $q$ is real-valued,
$c_{-j}=\overline{c_j}$.  Define
\[
    A(z)=z^L\sum_{j=-L}^{L}c_jz^j.
\]
Then $A$ is a polynomial of degree $2L$ with nonzero constant term, and
\begin{equation}\label{eq:self-inversive}
    A(z)=z^{2L}\overline{A(1/\overline z)}.
\end{equation}
Thus every root $\zeta$ is paired, with the same multiplicity, with
$1/\overline\zeta$.  There are no roots on the unit circle because
\[
    A(\e^{i\varphi})=\e^{iL\varphi}q(\varphi)\neq0.
\]
Exactly $L$ roots, counted with multiplicity, therefore lie outside the
unit disk.  Denote them by $\beta_1,\ldots,\beta_L$, and set
\[
    B(z)=\prod_{\ell=1}^{L}(z-\beta_\ell),
    \qquad
    B^\#(z)=z^L\overline{B(1/\overline z)}.
\]
The roots of $B^\#$ are the reciprocal conjugates of the $\beta_\ell$.
Hence $A$ and $BB^\#$ have the same roots and degree, so
$A=\kappa BB^\#$ for a nonzero constant $\kappa$.  On the unit circle,
$B^\#(z)=z^L\overline{B(z)}$, and consequently
\[
    q(\varphi)=\kappa\lvert B(\e^{i\varphi})\rvert^2.
\]
Strict positivity of $q$ forces $\kappa>0$.  Taking
$Q=\sqrt{\kappa}\,B$ proves the result.
\end{proof}

\begin{lemma}[Logarithmic second-harmonic bound]\label{lem:log-bound}
Let $N\geq1$, and let $P$ be a positive definite homogeneous polynomial of
degree $2N$ in two variables.  Define
\[
    p(\theta)=P(\cos\theta,\sin\theta),
    \qquad
    g(\theta)=\frac{1}{2N}\frac{p'(\theta)}{p(\theta)}.
\]
Then
\begin{equation}\label{eq:g-bound}
    \lvert\widehat g_2\rvert<1.
\end{equation}
\end{lemma}

\begin{proof}
The restriction of a degree-$2N$ homogeneous polynomial to the unit circle
has only even Fourier modes:
\[
    p(\theta)=\sum_{j=-N}^{N}c_j\e^{2ij\theta}.
\]
Since $p$ is $\pi$-periodic, the function $q(\varphi)=p(\varphi/2)$ is a
well-defined, strictly positive, $2\pi$-periodic trigonometric polynomial.
Let $L\leq N$ be its actual Fourier degree.  If $L=0$, then $q$, and hence $p$,
is constant, so $g=0$ and the result follows.  We may therefore assume
$L\geq1$.
Lemma~\ref{lem:factorization} gives
\begin{equation}\label{eq:q-factor}
    q(\varphi)=\lvert Q(\e^{i\varphi})\rvert^2,
    \qquad
    Q(z)=C\prod_{\ell=1}^{L}(1-\alpha_\ell z),
    \qquad \lvert\alpha_\ell\rvert<1.
\end{equation}
For $\lvert\alpha\rvert<1$, the logarithmic series
\[
    \log(1-\alpha\e^{i2\theta})
        =-\sum_{k=1}^{\infty}\frac{\alpha^k}{k}\e^{i2k\theta}
\]
and its differentiated series converge uniformly.  Define
\[
    G(\theta)=\frac{1}{2N}\log p(\theta)
             =\frac{1}{2N}\log q(2\theta).
\]
Using \eqref{eq:q-factor} and the logarithmic branches above gives
\[
\begin{aligned}
    G(\theta)=\frac{1}{2N}\bigg\{\log|C|^2
        +\sum_{\ell=1}^{L}\big[&\log(1-\alpha_\ell\e^{i2\theta})\\
        &+\log(1-\overline{\alpha_\ell}\e^{-i2\theta})\big]\bigg\}.
\end{aligned}
\]
Since $g=G'$, comparison of the second Fourier coefficients gives
\[
    \widehat g_2=-\frac{i}{N}\sum_{\ell=1}^{L}\alpha_\ell.
\]
It follows that
\[
    \lvert\widehat g_2\rvert
        \leq\frac{1}{N}\sum_{\ell=1}^{L}\lvert\alpha_\ell\rvert
        <\frac{L}{N}\leq1.
\]
\end{proof}

\subsection{Non-existence of Homogeneous Polynomial Lyapunov Functions}%
\label{sec:obstruction}

Suppose, towards a contradiction, that $P$ is a positive definite
homogeneous polynomial satisfying $\Lie P\leq0$.  Its degree is even, say
$2N$, because a homogeneous polynomial of odd degree changes sign under
$(x,y)\mapsto(-x,-y)$.  Set
\[
    p(\theta)=P(\cos\theta,\sin\theta)>0.
\]
Then $P(r,\theta)=r^{2N}p(\theta)$, and \eqref{eq:polar-system} yields
\begin{equation}\label{eq:LP-polar}
    \Lie P=r^{2N+2}
        \left[p'(\theta)+2N(-1+5\cos2\theta)p(\theta)\right].
\end{equation}
Define
\begin{equation}\label{eq:g-w}
    g(\theta)=\frac{1}{2N}\frac{p'(\theta)}{p(\theta)},
    \qquad
    w(\theta)=1-5\cos2\theta-g(\theta).
\end{equation}
The assumed Lie derivative inequality implies $w\geq0$.  Since
$g=(2N)^{-1}(\log p)'$ and $p$ is positive and $2\pi$-periodic,
\[
    \widehat g_0
        =\frac{1}{4\pi N}
          \left.\log p(\theta)\right|_{\theta=0}^{\theta=2\pi}
        =\frac{\log p(2\pi)-\log p(0)}{4\pi N}=0.
\]
Since $w=1-5\cos2\theta-g$ and $\cos2\theta$ has zero average,
$\widehat w_0=1$.  Using $w\geq0$, \eqref{eq:fourier-convention}, and
$|\e^{-2i\theta}|=1$ gives
\begin{equation}\label{eq:w-bound}
    \lvert\widehat w_2\rvert
        \leq\frac{1}{2\pi}\int_0^{2\pi}w(\theta)\,\dd\theta
        =\widehat w_0=1.
\end{equation}
Since $\widehat{\cos2\theta}_2=1/2$, taking the second Fourier coefficient
in \eqref{eq:g-w} gives
\[
    \widehat w_2=-\frac52-\widehat g_2.
\]
Therefore Lemma~\ref{lem:log-bound} and \eqref{eq:w-bound} imply
\[
    \frac52
        =\lvert\widehat w_2+\widehat g_2\rvert
        \leq\lvert\widehat w_2\rvert+\lvert\widehat g_2\rvert
        <2,
\]
which is impossible.  This proves Theorem~\ref{thm:main}(3) for homogeneous
polynomials of every degree, even under the weak derivative inequality.

\subsection{Non-existence of Local Real-Analytic Lyapunov Functions}%
\label{sec:local}

Suppose that a real-analytic function $V$ satisfies the local conditions in
Theorem~\ref{thm:main}(4).  Its Taylor expansion at the origin
has the form
\begin{equation}\label{eq:Taylor}
    V=P_m+P_{m+1}+\cdots,
\end{equation}
where each $P_j$ is homogeneous of degree $j$ and $P_m\neq0$
\cite[Ch.~2]{krantzParks2002}.  Evaluating
along a ray and letting $t\downarrow0$ gives
\[
    P_m(x)\geq0\qquad(x\in\R^2).
\]
Since $f$ is homogeneous of degree three,
\[
    \Lie V(tx)=t^{m+2}\Lie P_m(x)+o(t^{m+2}).
\]
The local derivative inequality therefore implies
\begin{equation}\label{eq:leading-Lie}
    \Lie P_m(x)\leq0\qquad(x\in\R^2).
\end{equation}

It remains to show that $P_m$ is positive definite.  Put
\[
    p_m(\theta)=P_m(\cos\theta,\sin\theta)\geq0.
\]
By the same calculation as in \eqref{eq:LP-polar},
\begin{equation}\label{eq:leading-angular}
    p_m'(\theta)+m(-1+5\cos2\theta)p_m(\theta)\leq0.
\end{equation}
If $p_m(\theta_0)=0$, define, for
$\theta\in[\theta_0,\theta_0+2\pi]$,
\[
    F(\theta)=
    \exp\!\left(
        m\int_{\theta_0}^{\theta}(-1+5\cos2s)\,\dd s
    \right)p_m(\theta).
\]
Then $F\geq0$, $F(\theta_0)=0$, and \eqref{eq:leading-angular} gives
$F'\leq0$.  Hence $F$ vanishes throughout the interval, contradicting
$P_m\neq0$.  Thus $p_m>0$, so $P_m$ is a positive definite homogeneous
polynomial satisfying \eqref{eq:leading-Lie}.  This contradicts
Theorem~\ref{thm:main}(3) and thus proves Theorem~\ref{thm:main}(4).

\section{Further Remarks}\label{sec:remarks}

\subsection{Regularity of Homogeneous Functions}

The loss of one derivative at the origin (see Lemma~\ref{lem:cone}) is sharp
in the following sense: there exist degree-$m$ homogeneous functions that are
$C^{m-1}$ without being polynomials, whereas $C^m$ regularity forces such a
function to be a homogeneous polynomial.

\begin{lemma}[Polynomiality of homogeneous functions]\label{lem:polynomiality}
Let $m$ be a positive integer.  If $U\colon\R^n\to\R$ is homogeneous of
degree $m$ and is $C^m$ in a neighbourhood of the origin, then
\begin{equation}\label{eq:homogeneous-taylor}
    U(x)=\sum_{|\alpha|=m}
        \frac{\partial^\alpha U(0)}{\alpha!}x^\alpha.
\end{equation}
In particular, $U$ is a homogeneous polynomial of degree $m$.
\end{lemma}

\begin{proof}
Fix $x\in\R^n$.  Taylor expansion at the origin gives
\[
    U(tx)=\sum_{|\alpha|\leq m}
        \frac{\partial^\alpha U(0)}{\alpha!}
        t^{|\alpha|}x^\alpha+o(t^m)
    \qquad(t\downarrow0).
\]
For $t>0$, homogeneity gives $U(tx)=t^mU(x)$.  The difference between the
Taylor polynomial above and $t^mU(x)$ is therefore $o(t^m)$.  Since this
difference is a polynomial in $t$ of degree at most $m$, all its coefficients
vanish.  Comparing the coefficient of $t^m$ gives
\eqref{eq:homogeneous-taylor}.
\end{proof}

The function $H$ is not $C^2$.  Otherwise Lemma~\ref{lem:polynomiality} would
make it a quadratic form.  However,
\begin{align*}
    H(1,0)&=H(0,1)=1,\\
    H(1,1)&=2\e^{-5}, & H(1,-1)&=2\e^5,
\end{align*}
which violates the parallelogram identity for quadratic forms.

More generally, for every positive integer $k$,
\begin{equation}\label{eq:Hk}
    H_k(r,\theta)=r^k h_k(\theta),
    \qquad
    h_k(\theta)=\exp\!\left(-\frac{5k}{2}\sin2\theta\right),
\end{equation}
satisfies, away from the origin,
\[
    \Lie H_k=-kr^2H_k,
    \qquad
    H_k\in C^{k-1}(\R^2)\cap C^\infty(\R^2\setminus\{0\}).
\]
It is not $C^k$.  Indeed, Lemma~\ref{lem:polynomiality} would make it a
homogeneous polynomial, whose restriction to the unit circle is a
trigonometric polynomial.  The angular profile $h_k$ defined in
\eqref{eq:Hk} satisfies
\[
    h_k'=-5k\cos2\theta\,h_k.
\]
If $h_k$ were a trigonometric polynomial with largest Fourier mode $M$,
the right-hand side would have a nonzero mode $M+2$, whereas $h_k'$ would
not.  Thus $H_k\in C^{k-1}(\R^2)\setminus C^k(\R^2)$, so the regularity
conclusion of Lemma~\ref{lem:cone} is sharp.

The same distinction appears in the weighted-homogeneous converse theorem
of Rosier: arbitrary prescribed finite regularity is available, whereas a
globally smooth homogeneous Lyapunov function would be polynomial
\cite[Thm.~2 and Rem.~(b)]{rosier1992}.

\subsection{Smooth Versus Real-Analytic Lyapunov Functions}

The real-analytic obstruction does not rule out smooth Lyapunov functions.
Set
\[
    W(x,y)=
    \begin{cases}
        \exp(-1/H(x,y)),&(x,y)\neq(0,0),\\
        0,&(x,y)=(0,0).
    \end{cases}
\]
Writing $H(r,\theta)=r^2h(\theta)$, where $h$ is smooth and uniformly
positive, shows that every derivative of $W$ away from the origin is a
finite sum of terms bounded by
$Cr^{-M}\exp(-c/r^2)$ for some $C,c>0$ and $M\geq0$.  All such terms tend
to zero as $r\downarrow0$.  Thus $W$ extends to a flat $C^\infty$ function
at the origin \cite[Lem.~2.20]{lee2013}, and
\[
    \Lie W=H^{-2}\exp(-1/H)\Lie H<0
\]
away from the origin.

\begin{remark}[Earlier nonhomogeneous counterexamples]
Bacciotti and Rosier gave a globally asymptotically stable planar degree-five
polynomial vector field with no local real-analytic Lyapunov function for irrational
parameter values, and hence with irrational coefficients
\cite[Prop.~5.2]{bacciottiRosier2005}.  Ahmadi and El Khadir later gave a
globally asymptotically stable planar degree-seven vector field with rational
coefficients and the same non-existence property
\cite{ahmadiElKhadir2018local}.
Both vector fields are nonhomogeneous; the present example is a homogeneous cubic
vector field with integer coefficients.
\end{remark}

\subsection{An Explicit Rational Lyapunov Function}\label{sec:rational}

A rational converse is known for asymptotically stable homogeneous vector
fields \cite[Thm.~III.1]{ahmadiElKhadir2020}.  For the present system, a
strict rational certificate is completely explicit.

Let
\[
    A=x^2-xy+y^2,
    \qquad
    B=x^2+xy+y^2.
\]
Both quadratic forms are positive definite, and
\[
    \frac12(x^2+y^2)\leq A,B\leq\frac32(x^2+y^2).
\]
Define $R(0,0)=0$ and, for $(x,y)\neq(0,0)$, let
\begin{equation}\label{eq:R-def}
    R(x,y)=(x^2+y^2)\left(\frac{A}{B}\right)^5.
\end{equation}
The preceding bounds give
\begin{equation}\label{eq:R-bounds}
    3^{-5}(x^2+y^2)\leq R(x,y)\leq3^5(x^2+y^2).
\end{equation}
It is therefore positive definite and radially unbounded.  Its angular
profile is smooth, so Lemma~\ref{lem:cone} shows that its extension at the
origin is $C^1$.

Logarithmic differentiation along the vector field and polynomial
simplification give
\begin{equation}\label{eq:R-Lie}
    \frac{\Lie R}{R}
    =-2\frac{x^6+7x^4y^2-3x^2y^4+y^6}
        {(x^2-xy+y^2)(x^2+xy+y^2)}.
\end{equation}
The denominator is $AB$ and is positive away from the origin.  If $y\neq0$,
set $t=x^2/y^2\geq0$.  The numerator then becomes
\[
    y^6(t^3+7t^2-3t+1).
\]
The quadratic $7t^2-3t+1$ is strictly positive on $\R$ because its
discriminant is $-19$, and $t^3\geq0$.  The case $y=0$ is immediate.
Thus \eqref{eq:R-Lie} is strictly negative away from the origin, so $R$ is
a strict rational Lyapunov function.

\subsection{A Two-Parameter Family of Counterexamples}\label{sec:family}

For $a>0$ and $b\in\R$, consider
\begin{align}
    \dot x&=-a(x^2+y^2)x+b(x^2-y^2)x-(x^2+y^2)y,
        \label{eq:family-x}\\
    \dot y&=-a(x^2+y^2)y+b(x^2-y^2)y+(x^2+y^2)x.
        \label{eq:family-y}
\end{align}
Its polar form is
\begin{equation}\label{eq:family-polar}
    \dot r=r^3(-a+b\cos2\theta),
    \qquad
    \dot\theta=r^2.
\end{equation}
The function
\begin{equation}\label{eq:Hab}
    H_{a,b}(r,\theta)=r^2\e^{-b\sin2\theta}
\end{equation}
satisfies, away from the origin,
\[
    \mathcal{L}_{f_{a,b}}H_{a,b}=-2ar^2H_{a,b}<0.
\]
Lemma~\ref{lem:cone} and the bounds
\[
    \e^{-|b|}r^2\leq H_{a,b}\leq\e^{|b|}r^2
\]
show that $H_{a,b}$ is $C^1$, positive definite, and radially unbounded.
Together with the strict derivative identity, these properties imply by
Lyapunov's direct method that the origin is globally asymptotically stable
for every $a>0$.

Now suppose a positive definite homogeneous polynomial $P$ satisfies
$\mathcal{L}_{f_{a,b}}P\leq0$.  With $g$ as in
Lemma~\ref{lem:log-bound},
\[
    w(\theta)=a-b\cos2\theta-g(\theta)\geq0.
\]
Thus $\widehat w_0=a$, $\lvert\widehat w_2\rvert\leq a$, and
\[
    \widehat w_2=-\frac b2-\widehat g_2.
\]
Lemma~\ref{lem:log-bound} yields the necessary condition
\[
    \frac{|b|}{2}
        \leq\lvert\widehat w_2\rvert+\lvert\widehat g_2\rvert
        <a+1.
\]
We have proved the following family version.

\begin{theorem}[Family of counterexamples]\label{thm:family}
If
\[
    a>0,
    \qquad
    |b|\geq2(a+1),
\]
then \eqref{eq:family-x}--\eqref{eq:family-y} has a globally
asymptotically stable origin but admits no positive definite homogeneous
polynomial $P$ satisfying $\mathcal{L}_{f_{a,b}}P\leq0$.
\end{theorem}

The original system \eqref{eq:field-x}--\eqref{eq:field-y} corresponds to $(a,b)=(1,5)$.  In particular, the strict
region $a>0$ and $|b|>2(a+1)$ is an open set of counterexamples within this
two-parameter family.

\subsection{Lean 4 Formalization}

A machine-checked Lean~4 formalization of all four items of
Theorem~\ref{thm:main} is available at
\url{https://github.com/j49liu/homogeneous-lyapunov-counterexample-lean}.
The development closely follows the proof architecture presented here while keeping selected modules reusable.

\section{Conclusion}

The planar cubic system \eqref{eq:field-x}--\eqref{eq:field-y} is globally
asymptotically stable and has two explicit strict homogeneous Lyapunov functions,
one involving an exponential and one rational.  Nevertheless,
the Fourier constraint in Subsection~\ref{sec:obstruction} rules out
homogeneous polynomial candidates of every degree, even under a weak Lie
derivative inequality.
The leading-term argument then excludes all local real-analytic Lyapunov
functions.  The mechanism persists on an open region
of the two-parameter family \eqref{eq:family-x}--\eqref{eq:family-y}.
These results give a negative answer to the homogeneous polynomial Lyapunov
converse conjecture posed in \cite{ahmadi2012difficulty,ahmadiParrilo2013}.

\section*{Acknowledgment}

The authors used \mbox{OpenAI} GPT-5.6-Sol Pro to discover the counterexample
and produce initial versions of the mathematical arguments in
Sections~II--IV.  \mbox{OpenAI} Codex (GPT-5.6-Sol Ultra) assisted with
drafting the manuscript and developing the Lean~4 formalization of the main
result.  The authors independently verified all AI-generated content and take
full responsibility for the article.

\bibliographystyle{IEEEtran}
\bibliography{references}


\end{document}